\documentclass[12pt]{amsart}

\usepackage{amssymb}
\usepackage{amsmath}
\usepackage{amsfonts}
\usepackage[all]{xy} 
\usepackage{graphicx} 
\usepackage{color} 

\newcommand{\disk}{\ensuremath{\mathbb{D}} } 
\newcommand{\cdisk}{\ensuremath{\overline{\mathbb{D}}}} 
\newcommand{\pcdisk}{\ensuremath{\overline{\mathbb{D}}_0}} 

\newcommand{\pcdiski}{\ensuremath{\overline{\mathbb{D}}^*_{\infty}}}

\newcommand{\riem}{\Sigma}

\newcommand{\Aut}{\operatorname{Aut}}

\newcommand{\pmcgi}[1][]{\mathrm{PModI}({#1})} 
\renewcommand{\Bbb}[1]{\ensuremath{\mathbb{#1}}}
\newcommand{\ann}{\mathcal{A}}  

\newcommand{\annold}{\mathbb{A}}

\newcommand{\cpunc}{\mathbb{C}^*} 
\newcommand{\Oqc}{\mathcal{O}_{\mathrm{qc}}} 

\newcommand{\mba}{\widetilde{\mathcal{M}}(0,2)}

\newcommand{\cbar}{\overline{\mathbb{C}}}

\theoremstyle{plain}
        \newtheorem{theorem}{Theorem}[section]
        \newtheorem{lemma}[theorem]{Lemma}
        \newtheorem{proposition}[theorem]{Proposition}
        \newtheorem{corollary}[theorem]{Corollary}

\theoremstyle{definition}
        \newtheorem{definition}[theorem]{Definition}

\theoremstyle{remark}
    \newtheorem{remark}[theorem]{Remark}

\numberwithin{equation}{section} 
\numberwithin{figure}{section} 

\usepackage{fullpage}

\begin{document}

\title{The semigroup of rigged annuli and the Teichm\"uller space
of the annulus}

\author{David Radnell}
\address{Department of Mathematics and Statistics \\
American University of Sharjah \\
PO BOX 26666, Sharjah, UAE}
\email[D. ~Radnell]{dradnell@aus.edu}

\author{Eric Schippers}
\address{Department of Mathematics \\
University of Manitoba\\
Winnipeg, MB, R3T 2N2, Canada}
\email[E.
~Schippers]{eric\_schippers@umanitoba.ca}


\date{January 28, 2010}


\begin{abstract} Neretin and Segal independently
 defined a semigroup of annuli with boundary
 parametrizations, which is viewed as a complexification of the group
 of diffeomorphisms of the circle.  By extending the parametrizations
 to quasisymmetries, we show that this semigroup is a quotient of
 the Teichm\"uller space of doubly-connected Riemann surfaces by a
 $\mathbb{Z}$ action.  Furthermore, the semigroup can be given a
 complex structure in two distinct, natural ways.
 We show that these two complex structures are equivalent, and
 furthermore that multiplication is holomorphic.  Finally, we show
 that the class of quasiconformally-extendible conformal maps of the
 disk to itself is a complex submanifold in which composition is
 holomorphic.
\end{abstract}

\maketitle

\begin{section}{Introduction} \label{se:Introduction}
\begin{subsection}{Motivation and statement of results}
 The Lie algebra of the group of diffeomorphisms of the circle
 $\mathrm{Diff}(S^1)$ is the Witt algebra.
 It has been known for some time that there is no Lie group  whose Lie algebra is the complexification of
 the Virasoro algebra or Witt algebra (see Lempert \cite{Lempert}
 for a proof).  Thus there is no Lie group which is the complexification of
 $\mathrm{Diff}(S^1)$.
 However, Segal \cite{Segal} and Neretin \cite{Neretin_semigroup, Neretin_holomorphic} independently defined a semigroup which
 is in some sense the desired complexification.  The Neretin-Segal semigroup is defined as follows.
 The term \textit{annulus} will refer to a bordered Riemann surface that is biholomorphically equivalent to a doubly-connected domain in $\mathbb{C}$.
 Consider the set of annuli $A$, together with
 parametrizations $\phi_i:S^1 \rightarrow \partial_i A$, $i=1,2$, of
 each boundary component $\partial_i A$.   The maps $\phi_1 $ and $\phi_2$ are respectively orientation reversing and orientation preserving.
 Two such annuli with parametrizations are equivalent if
 there is a biholomorphism between them which preserves the parametrizations. The multiplication is obtained by
 sewing the first boundary component of one annulus to the second boundary component of another, by identifying points using the corresponding parametrizations, and carrying along
 the data of the remaining two parametrizations.  From here on we
 will refer to an element of this semigroup as a rigged annulus, where the
 term ``rigging'' refers to the boundary parametrizations.

 It is customary in conformal field theory (as defined by Segal \cite{Segal} and Kontsevich) to choose these
 parametrizations to be either diffeomorphisms or diffeomorphisms
 with analytic continuations to an annular neighbourhood of the
 boundary.  We choose rather quasisymmetric boundary
 parametrizations.  As a consequence, we are able to prove the
 following facts. \newpage

\noindent {\bf Results}:
\begin{enumerate}
\item The quasisymmetric Neretin-Segal semigroup is a quotient of
the
  Teichm\"uller space of the annulus by a properly discontinuous,
  fixed-point-free $\mathbb{Z}$-action.  The semigroup of rigged annuli
  thus inherits a complex structure from the Teichm\"uller space of the
  annulus. (Theorem
  \ref{th:TA_is_M}, Theorem \ref{th:complex_structure_from_TA} and diagram (\ref{teepee})).
\item The quasisymmetric Neretin-Segal semigroup is a complex
Banach manifold,
  locally modelled on certain function spaces in a natural way. (Theorem \ref{th:complexstructure_via_Oqc}).

\item The complex structures in the two previous item are compatible.
(Theorem \ref{th:L_holomorphic}, Theorem
\ref{th:Linverse_holomorphic} and diagram (\ref{teepee})).

\item Multiplication is holomorphic. (Corollary
\ref{co:multiplication_holomorphic}).

 \item The set of quasiconformally extendible one-to-one holomorphic maps of the
 disk into itself is a complex submanifold of the semigroup of
 rigged annuli, in which composition is holomorphic. (Theorem
 \ref{th:Esubmanifold} and Corollary \ref{co:Emultiplication_holo}).
\end{enumerate}
 Note that, in items (4) and (5), we prove holomorphicity in the
 sense that the derivative approximates the function up to first
 order; this is a much stronger result than G\^ateaux holomorphicity.
 Note also that result (1) establishes that the Teichm\"uller space of the
 annulus modulo a $\mathbb{Z}$ action possesses a semigroup
 structure.

 Two points must be emphasized.  First, without the choice of
 quasisymmetric riggings, it is impossible to establish the relation
 between the Neretin-Segal semigroup and the Teichm\"uller space. Second, the quasiconformal
 Teichm\"uller space of the annulus is infinite-dimensional,
 and in fact contains the information of the
 parametrizations.  When Teichm\"uller space appeared in previous
 models of the rigged moduli spaces of Riemann surfaces of arbitrary type, it was the
 finite-dimensional Teichm\"uller space of a compact surface.  It appeared as a base space of a fiber
  space, whose infinite-dimensional fibers consisted of the
  riggings.  The fact that the information of the riggings is somehow contained
  in the Teichm\"uller space of a bordered surface was
  demonstrated in \cite{RadSchip05}.

 Using these two insights,  in previous work the authors demonstrated
 the general relation between the
 rigged moduli spaces and quasiconformal Teichm\"uller space
 \cite{RadSchip05}.  Although the result (1) above has never been
 published, it is an immediate consequence of this previous work.
 The results (2) and (3) are not, and require some comment.  In \cite{RadSchip_fiber} we
  used the idea of the
 fiber space described in the previous paragraph to demonstrate that the Teichm\"uller space of
 a bordered Riemann surface has a natural fiber structure, with the
 fibers consisting of non-overlapping maps into the Riemann surface.
 This space of non-overlapping maps possesses
 a complex structure in an independent way (strangely, also related
 to the universal Teichm\"uller space)
 \cite{RadSchip_nonoverlapping}.  However, in the case of the
 annulus, there is a continuous family of conformal automorphisms, and consequently
  our previous results on the fiber structure do not apply.
 The same is true for our proofs of the compatibility of the two complex
 structures. Thus
 proofs are necessary in the doubly-connected case,
 and providing them is the main purpose of this paper.
 On the other hand, in some
 ways the proof in this special case is more transparent (see
 Remark \ref{re:why_easier} ahead).

 There is growing recognition of the advantages
 of using quasisymmetries of the circle rather than
 diffeomorphisms in the literature
 (e.g. \cite{NagSullivan, Pekonen, TTmem}).
 The quasisymmetric version of the Neretin-Segal semigroup itself appears in
 Pickrell \cite{Pickrell}.  We will continue to refer to the semigroup with
 quasisymmetric riggings as the Neretin-Segal semigroup.

 In the rest of Section \ref{se:Introduction}, we define the moduli space of rigged
 annuli in the quasisymmetric setting.  Section
 \ref{se_top:puncture_model} outlines the alternate model of the
 Neretin-Segal semigroup in terms of non-overlapping mappings into
 the sphere, which will be the model used throughout the paper.  In
 this section we also prove the multiplication formula in the
 quasisymmetric setting, and endow the semigroup with a complex
 structure.  Section \ref{se:relation_with_Teich} proves that the
 Neretin-Segal semigroup is in one-to-one correspondence with the
 Teichm\"uller space of the annulus modulo a $\mathbb{Z}$ action,
 and thus inherits a complex structure.
 The main results, the equivalence of the two complex structures, is the subject of
 Section \ref{se:compatibility}.  In this section we also
 demonstrate the holomorphicity of multiplication and conclude with
 some consequences for the semigroup of bounded univalent
 functions with quasiconformal extensions.
\end{subsection}
\begin{subsection}{Sewing Riemann surfaces via quasisymmetries}

 In this section we define quasisymmetric boundary parametrizations and sewing
 of general Riemann surfaces. Details and proofs can be
 found in \cite{RadSchip05}.

In the following, the term ``bordered'' Riemann surface refers to a
Riemann surface with boundary in the standard sense (see e.g.
\cite{AhlforsSario}).  That is, there is an atlas for the Riemann
surface such that each point of the boundary is contained in the
domain of a chart onto a relatively open subset of the closed upper
half plane, which takes the boundary to a finite open subinterval
$(a,b)$ of the real line, and furthermore the overlap maps of the
atlas are holomorphic on their interiors.

\begin{definition}  We say that $\riem$ is a bordered Riemann
 surface of type $(g,n)$ if it is a bordered Riemann surface
 such that (1) its boundary consists of $n$ ordered closed
 curves homeomorphic to $S^1$ and (2) it is biholomorphically equivalent to
 a compact Riemann surface of genus $g$ with $n$ simply-connected
 non-overlapping regions, each biholomorphic to a disk,
 removed.
\end{definition}

We denote the boundary of $\riem$ by $\partial \riem$ and the $i$th
boundary component by $\partial_i \riem$. A map $S^1 \to \partial_i
\riem$ is called a \textit{boundary parametrization} or
\textit{rigging}. Following \cite{RadSchip05}, the class of such mappings will be quasisymmetric
as defined below.

\begin{definition} \label{quasisymmetricline}
 An (orientation preserving)
 homeomorphism
$$h:\overline{\mathbb R} \rightarrow
 \overline{\mathbb R}
$$ is $k$-quasisymmetric if there exists a
 constant $k$ such that
 \[  \frac{1}{k} \leq \frac{h(x+t)-h(x)}{h(x)-h(x-t)} \leq k  \]
 for all $x,\,t \in \overline{\mathbb R}$.  If $h$ is
 quasisymmetric for some unspecified $k$ it is simply called
 quasisymmetric.
\end{definition}

We find it more convenient to work on $S^1$ than on
$\overline{\mathbb R}$.  It is also necessary to speak of
quasisymmetry of a mapping on a closed boundary curve of a Riemann
surface. The map $T(z) = i(1+z)/(1-z)$ sends the unit circle to
$\overline{\Bbb{R}}$ with $T(1) = \infty$.

For $r\neq 1$, let $\annold(r) \subset \mathbb{C}$ be the annulus bounded by circles of radius $1$ and $r$.
\begin{definition} \label{quasisymmetriccircle}
 Let $h: S^1 \rightarrow S^1$ be a homeomorphism.
 \begin{enumerate}
 \item Let $e^{i\theta}$ be chosen so that $e^{i\theta}h(1)=1$.
Then we say that $h$ is quasisymmetric
 if $T \circ e^{i\theta} h \circ T^{-1}$ is quasisymmetric according to
 Definition \ref{quasisymmetricline}.
 \item Let $C$ be a connected component of the boundary of a bordered Riemann
 surface, $h$ be a homeomorphism of $C$ into $S^1$, and $\annold_C$ be an annular neighbourhood of $C$.  We say that
 $h$ is quasisymmetric if, for any biholomorphism $F: \annold_C
 \rightarrow \annold(r)$, $h \circ F^{-1}$ is
 quasisymmetric on $S^1$ in the sense of part one.
 \end{enumerate}
\end{definition}
Note that if $r<1$ then the above map is orientation preserving,
otherwise it is orientation reversing.

A map is quasisymmetric if and only if it is the boundary value of a
quasiconformal map on a collared neighbourhood of the boundary. This
follows from the Ahlfors-Beurling extension theorem \cite{Lehto}.

 We now describe the sewing of arbitrary Riemann surfaces.
Let $\riem_1$ and $\riem_2$ be bordered Riemann surfaces of type
$(g_1,n_1)$ and $(g_2,n_2)$ respectively where $n_1 > 0 $ and
$n_2>0$.  Let $C_1$ and $C_2$ each be a boundary component of
$\riem_1$ and $\riem_2$ respectively, and let $\psi_i$ be
oppositely oriented quasisymmetric parametrizations of $C_i$ (that
is, quasisymmetric maps $\psi_i: S^1 \to C_i$  for $i=1,2$, in the
sense of Definition \ref{quasisymmetriccircle} with say $r_1>1$
and $r_2 <1$). l. Note that $\psi_2 \circ \psi_1^{-1} : C_1 \to
C_2$ is an orientation reversing map.

Let $\riem_1 \# \riem_2 = \riem_1 \sqcup \riem_2 / \sim$ where $x
\sim y$ if and only if $x \in C_1$, $y \in C_2$ and $(\psi_2 \circ
\psi_1^{-1})(x) = y$.  $C_1$ and $C_2$ correspond to a common curve
on $\riem_1 \# \riem_2$.

\begin{theorem}{\cite[Theorems 3.2 and 3.3]{RadSchip05}}
\label{th:sewing} There is a unique complex structure on $\riem_1 \#
\riem_2$ which is compatible with the original complex structures on
$\riem_1$ and $\riem_2$.
\end{theorem}
The proof of this theorem is based on conformal welding.
\end{subsection}
\begin{subsection}{The moduli space of rigged annuli}

In this section we describe the Neretin-Segal semigroup of rigged
annuli.   In the rest of the paper we will use another model of
this moduli space, which will be described in Section
\ref{se:puncture_model_annuli}. The model in this section is
included because it is more immediately understandable (this is
especially true of the multiplication). Both models are well-known
in conformal field theory to be equivalent, but we must establish
this rigorously in the quasisymmetric setting. This will be done
in the next section.

\begin{definition}
\label{de:M02}
 Consider the set of ordered pairs $(A,\phi)$
 where
 \begin{enumerate}
 \item $A$ is a bordered Riemann surface of type $(0,2)$ (i.e.
 doubly-connected) with boundary curves $\partial_i A$, $i=1,2$ and
 \item $\phi=(\phi_1,\phi_2)$ where $\phi_i:S^1 \rightarrow
 \partial_i A$ are quasisymmetries that are respectively orientation reversing and preserving.
 \end{enumerate}
 We define
 \[  \widetilde{\mathcal{M}}(0,2) =\{(A,\phi)\}/\sim  \]
 where $(A,\phi) \sim (B,\psi)$ if there exists a biholomorphism
 $\sigma:A \rightarrow B$ such that $\sigma \circ \phi = \psi$.
 We denote equivalence classes by $[A,\phi]$.
\end{definition}

We have made a slight but fundamental change to the standard
definition: the boundary parametrizations are quasisymmetries. As
was mentioned in the introduction, our choice makes it possible to
connect the moduli space of rigged annuli to the Teichm\"uller space
of the annulus. It is not possible to do this with diffeomorphisms
or analytic diffeomorphisms.

\begin{remark}
 The rigged moduli space of annuli is a special case of a more
 general concept from conformal field theory, that of the rigged
 Riemann surface.  Given a bordered Riemann surface $\Sigma$ of
 type $(g,n)$,
 we denote the ordered set of quasisymmetric boundary
 parametrizations  by $\psi = \left( \psi^1, \ldots,
 \psi^n \right)$. The pair $(\riem, \psi)$ is called a
 rigged Riemann surface.
We define an equivalence relation on the set $\{(\riem,\psi)\}$ of
type $(g,n)$ rigged Riemann surfaces: $(\riem_1, \psi_1) \sim
(\riem_2, \psi_2)$ if and only if there exists a biholomorphism
$\sigma :\riem_1 \to \riem_2$ such that $\psi_2 =  \sigma \circ
\psi_1 $. The moduli space of rigged Riemann surfaces is
 \[ \widetilde{\mathcal{M}}(g,n) = \{ (\riem, \psi )\} / \sim.
 \]
\end{remark}

\begin{remark}
\label{re:in_and_out} In \cite{RadSchip05} all boundary
parametrizations are positively oriented and each boundary component
is specified as \textit{incoming} or \textit{outgoing} by an
assignment of the symbol $-$ or $+$ respectively. In the current
setting this information is specified by the choice of orientation
of each rigging, with orientation preserving corresponding to $+$.
Hence $\mba$ corresponds to $\widetilde{\mathcal{M}}^B(g, n^-, n^+)$
with $g = n^- = n^+ = 1$ in \cite[Definition 5.3]{RadSchip05}.

\end{remark}

 Next, we describe the multiplication operation in
 $\mba$.
 \begin{definition} \label{de:border_multiplication}
  The product of two elements $[A,\phi], [B,\psi]$ of $\mba$ is
  the rigged annulus
  \[  [A,\phi] \times [B,\psi] = [A \# B,\rho], \]
  where $A$ and $B$ are sewn together along $\partial_1 A$ and $\partial_2 B$
  via $\psi_2 \circ \phi_1^{-1}$ and $\rho = (\psi_1,\phi_2)$.
 \end{definition}

 \begin{remark}
  One may find a uniformizing biholomorphism $G: A \# B
  \rightarrow D$ onto an annulus $D \subset \mathbb{C}$; in that case, the
  joining curve will map to a quasicircle and the new boundary
  parametrization will be $(G \circ \rho_1, G \circ \rho_2)$.
 \end{remark}
\end{subsection}
\end{section}
\begin{section}{The non-overlapping mapping model of the semigroup of
 rigged annuli} \label{se_top:puncture_model}
\begin{subsection}{Non-overlapping mappings into the Riemann sphere}
\label{se:puncture_model_annuli}
 In this section, we give an alternate definition of the moduli
 space of rigged annuli, in terms of non-overlapping mappings into
 the Riemann sphere.  The equivalence of the two models
 (in the sense that there is a one-to-one correspondence)
 is well-known \cite{Neretin_holomorphic, Segal} in the case of
 analytic or diffeomorphic parametrizations.   We establish here the
 equivalence with our choice of
 riggings.  The proof of the equivalence relies on the technique of
 conformal welding.

 We must first  choose the correct analytic
 conditions on the set of non-overlapping mappings, to match the
 quasisymmetric riggings.  The obvious choice is the set of
 univalent maps with quasiconformal extensions.  These restrict to
 quasisymmetries on the boundary, and conversely by the Ahlfors-Beurling
 extension theorem any quasisymmetry
 is the boundary value of such a quasiconformally extendible
 mapping.

 We now give the precise description of the alternate
 model.
 Let $\mathbb{D}=\{z\,:\,|z|<1\}$ and $\mathbb{D}^*=\{z\,:\,|z|>1\}
\cup \{\infty\}$.
 Note that the boundary $\partial \mathbb{D}^*$ is $S^1$ with clockwise orientation.

 \begin{definition}
 \label{de:Ao}
  Let  $\mathcal{A}^o = \{(f,g)\}$ where $f:\mathbb{D} \rightarrow \mathbb{C}$,
  $g:\mathbb{D}^* \rightarrow \overline{\mathbb{C}}$ are one-to-one holomorphic maps satisfying
  \begin{enumerate}
    \item $f$ has a quasiconformal extension to $\mathbb{C}$ and $g$ has a quasiconformal extension to $\overline{\mathbb{C}}$.
   \item $f(\overline{\mathbb{D}}) \cap g(\overline{\mathbb{D}^*}) = \emptyset$
   \item $f(0)=0$
   \item $g(\infty)=\infty$, $g'(\infty)=1$
  \end{enumerate}
 \end{definition}

 It is advantageous to consider an enlargement of this set of annuli, to
 include ``degenerate'' annuli whose two boundary components might touch.
 It is difficult to express this enlargement in terms of Riemann surfaces
 with boundary parametrizations.  However it is easily expressed
 in terms of non-overlapping holomorphic maps.
 \begin{definition}
 Let  $\mathcal{A} = \{(f,g)\}$ where $f:\mathbb{D} \rightarrow \mathbb{C}$, $g:\mathbb{D}^* \rightarrow \overline{\mathbb{C}}$ are one-to-one holomorphic maps satisfying
  \begin{enumerate}
    \item $f$ has a quasiconformal extension to $\mathbb{C}$ and $g$ has a quasiconformal extension to $\overline{\mathbb{C}}$.
   \item $f(\mathbb{D}) \cap g(\mathbb{D}^*) = \emptyset$
   \item $f(0)=0$,
   \item $g(\infty)=\infty$, $g'(\infty)=1$
  \end{enumerate}
 \end{definition}

 This extension of the rigged annuli has several advantages: first,
 it has an identity element (it is a monoid), second, it contains
 the group of quasisymmetries of $S^1$.

In order to show the correspondence between $\mba$ and the
Teichm\"uller space of the annulus, it is necessary to describe
the operation of sewing Riemann surfaces using quasisymmetric
maps.  We do this now. Define the punctured closed disks $\pcdisk
= \{z\,:\,0<|z|\leq1\}$ and $\pcdiski =
\{z\,:\,1\leq|z|<\infty\}$, considered as subsets of $\cbar$, and
let $\mathbb{D}_0$ and $\mathbb{D}_\infty$ denote their respective
interiors.

Given an annulus $A$ we can obtain a twice punctured genus-zero
Riemann surface in the following way. For details and the general
case of type $(g,n)$ surfaces, see \cite[Section 3]{RadSchip05}. Let
$[A,\tau] \in \mba$ (see Definition \ref{de:M02}) where
$\tau=(\tau_0,\tau_{\infty})$, and  $\tau_0:\partial \mathbb{D}
\rightarrow \partial_1 A$ and $\tau_{\infty}:\partial \mathbb{D}^*
\rightarrow \partial_2 A$ are fixed quasisymmetric mappings (see
Definition \ref{quasisymmetriccircle}). We sew on the punctured
disks $\pcdisk$ and $\pcdiski$ as follows.

Consider the disjoint union of $A$, $\pcdisk$ and $\pcdiski$.
Identifying boundary points using  $\tau$, the result is a compact
surface $\riem^P$ with punctures $p_0$ and $p_1$ corresponding to
the punctures $0$ and $\infty$ of $\pcdisk$ and $\pcdiski$,
respectively. That is, let
$$
\riem^P = (A \sqcup \pcdisk \sqcup \pcdiski ) /\sim
$$
where $p \sim q$ if and only if $p \in \partial_1 A$, $q \in
\partial \mathbb{D}$, and $p=\tau_0(q)$, or $p \in \partial_2 A$, $q \in
\partial \mathbb{D}^*$ and $p=\tau_\infty(q)$.   By Theorem \ref{th:sewing}, $\riem^P$
has a unique complex structure which is compatible with that of
both $A$ and the disks $\pcdisk$ and $\pcdiski$. If $\riem^P$ is
obtained from $A$ in this way we will say that $\riem^P$ is
obtained by ``sewing caps on $A$ via $\tau$'' and we write
\begin{equation}
\label{riemP} \riem^P = A \#_{\tau} (\pcdisk \sqcup \pcdiski) .
\end{equation}
The parametrizations $\tau_0$ can be extended continuously to a
map $\tilde{\tau}_0: \pcdisk \to \riem^P$ to the caps of $\riem^P$
by
\begin{equation}
\label{eq:tauextension_definition} \tilde{\tau}_0(x) =
\begin{cases}
\tau_0(x) , & \text{for }  x \in \partial \disk \\
x , & \text{for } x \in \disk_0 .
\end{cases}
\end{equation}
This map is a biholomorphism on $\disk_0$ and has a quasiconformal
extensions to a neighbourhood of $\pcdisk$. Similarly,
$\tau_{\infty}$ can be extended to a map $\tilde{\tau}_{\infty}:
\pcdiski \to \riem^P$.  They can also be extended analytically
across $0$ and $\infty$.

We can now make the identification between $\mathcal{A}^0$ and
$\mba$ as follows.  Let $[A,\tau] \in \mba$ have representative
$(A,\tau)$. Sew on caps following the procedure above to obtain the
triple $(\riem^P,\tilde{\tau}_1,\tilde{\tau}_2)$.  Since $\riem^P$
is a genus zero Riemann surface with two punctures, there exists a
biholomorphism $H: \riem^P \rightarrow \cpunc$.  There is a unique
such $H$ such that the conformal extension $g$ of $H \circ
\tilde{\tau}_\infty$ satisfies $g'(\infty)=1$.  We then define
\begin{equation} \label{eq:puncture_border_identification}
 R([A,\tau])= (f,g)
\end{equation}
where $g$ is the analytic extension of $\left. H \circ
\tilde{\tau}_\infty \right|_{\mathbb{D}^*}$ across $\infty$ and $f$
is the analytic extension of $\left. H \circ
\tilde{\tau}\right|_{\mathbb{D}}$ across $0$.

\begin{theorem}
\label{th:mba_Ao_iso}
$R:\mba \rightarrow \mathcal{A}^0$ is a well-defined, one-to-one, onto map.
\end{theorem}
\begin{proof}  Although a direct proof can be given fairly easily,
we refer to \cite[Theorem 5.1]{RadSchip05} with $g=1$, $n^-=1$ and
$n^+=1$ for the sake of brevity (see Remark \ref{re:in_and_out}).
\end{proof}

\end{subsection}

\begin{subsection}{The sewing equations and multiplication}
 We now give the formula for multiplication in the non-overlapping
 mapping model.   This formula was obtained by Huang
 \cite{Huang}.  It is necessary here to give a proof for the class
 of quasisymmetric riggings, which can be accomplished with the technique of conformal welding.  This introduces no
 difficulties and is no more tedious than the original procedure for sewing,
 although it does require a deeper uniformization result (ultimately
 relying on the measurable Riemann mapping theorem).

 We will need only the following theorem.  It is a standard
 fact, although usually presented with a different choice of
 normalizations.  Since the normalizations are of some importance,
 we include a proof.
 \begin{theorem}[conformal welding] \label{th:conformalwelding}
  Fix $a \in \mathbb{C} \backslash \{0\}$.
  Let $\phi:S^1 \rightarrow S^1$ be a quasisymmetric mapping.  There exists a
  unique pair $(F,G)$, such that $F:\mathbb{D} \rightarrow \mathbb{C}$ and
  $G:\mathbb{D}^* \rightarrow \overline{\mathbb{C}}$ are one-to-one holomorphic maps with
  quasiconformal extensions to $\overline{\mathbb{D}}$ satisfying
  \begin{enumerate}
   \item $F(\partial \mathbb{D})=G(\partial \mathbb{D}^*)$ as sets
   \item $F(0)=0$, $G(\infty)=\infty$ and $G'(\infty)=a$
   \item $\phi = \left. G^{-1} \circ F \right|_{S^1}$.
  \end{enumerate}
 \end{theorem}
 \begin{proof}
  Let $w_\mu:\mathbb{D}^* \rightarrow \mathbb{D}^*$ be a
  quasiconformal extension of $\gamma$, in the sense that it
  extends homeomorphically to $\overline{\mathbb{D}^*}$ and
  satisfies $\left. w_\mu \right|_{S^1}=\phi$.  Let $\mu$
  be the complex dilatation of this extension.  Such a
  quasiconformal extension exists by the Ahlfors-Beurling
  extension theorem \cite{Lehto}.  However, it is not unique.

  Let $w^\mu:\overline{\mathbb{C}} \rightarrow
  \overline{\mathbb{C}}$ be the unique quasiconformal map with dilatation
  $\mu$ on $\mathbb{D}^*$ and $0$ on $\mathbb{D}$, satisfying the
  normalization $w^\mu(0)=0$, $w^\mu \circ {w_\mu}^{-1}(\infty) =
  \infty$, $(w^\mu \circ {w_\mu}^{-1})'(\infty) = a$.  ($w^\mu$ is
  unique in the sense that it is fixed by $\mu$ and these
  normalizations).  Note that $w^\mu \circ {w_\mu}^{-1}$ is
  holomorphic.  Now set
  \[  F = \left. w^\mu \right|_{\mathbb{D}} \ \ \ \mbox{and} \ \ \
      G = \left. w^\mu \circ w_\mu^{-1} \right|_{\mathbb{D}^*}.  \]
  Both maps are holomorphic on their domains, have quasiconformal
  extensions, and clearly $\left. G^{-1} \circ F \right|_{S^1}= \left. w_\mu
  \right|_{S^1}= \phi$.  It is also clear that $F$ and $G$ satisfy
  the desired normalizations.

  Now we show that $F$ and $G$ are uniquely determined by $\phi$.
   Given $\phi \in QS(S^1)$ define
  $\mathcal{T}(\phi)=(F,G)$ where $F$ and $G$ are given
  by the above procedure.  First we show that $\mathcal{T}$ is
  well-defined.  Say $w_{\mu}$ and $w_\nu$ are two different
  quasiconformal extensions of $\phi$ with dilatations $\mu$ and
  $\nu$ respectively.  Let
 $$
 S(z) = \begin{cases}
  w^\mu \circ ({w^\nu})^{-1}
     (z), & \text{if }  z \in w^\nu(\mathbb{D}) \\
     w^\mu \circ (w_{\mu})^{-1}
     \circ w_{\nu} \circ (w^{\nu})^{-1}(z), & \text{if }   z \in w^\nu(\mathbb{D^*}).
 \end{cases}
 $$

  $S$ is quasiconformal on each piece, and extends to a one-to-one
  continuous map of $\overline{\mathbb{C}}$ since ${w_\mu}^{-1}
  \circ w_\nu$ is the identity on $S^1$.  Thus $S$ is
  quasiconformal on $\overline{\mathbb{C}}$, by removability
  of quasicircles \cite[V.3]{LV}.  Since the dilatation
  of $S$ is zero on each piece, $S$ is in fact conformal.  It is
  easily checked that $S(0)=0$, $S(\infty)=\infty$ and
  $S'(\infty)=1$, so $S(z)=z$.  In particular $w^\mu=w^\nu$ on
  $\mathbb{D}$ and $w^\mu \circ {w_\mu}^{-1} = w^\nu \circ
  {w_\nu}^{-1}$ on $\mathbb{D}^*$.

  Denote the potential inverse of $\mathcal{T}$ by
  $\mathcal{S}(F,G)=G^{-1} \circ F$.  The second paragraph of
  the proof shows
  that $\mathcal{S}$ is surjective.  We need to show that
  $\mathcal{S}$ is injective; to do this we show that $\mathcal{T} \circ
  \mathcal{S}$ is the identity map.  Given $(F,G)$ satisfying the
  required normalization, let $F^\mu$ be a quasiconformal
  extension of $F$ to the sphere with dilatation $\mu$, say.  Then
  $w_\mu=G^{-1} \circ F^\mu$ is a quasiconformal extension of
  $G^{-1} \circ F$ to $\mathbb{D}^*$, and the corresponding
  $w^\mu$ is clearly $F^\mu$.  Thus $\mathcal{T} \circ \mathcal{S}
  (F,G)=(F,G)$.
 \end{proof}

 Next we give the formula for multiplication of annuli.
 The notation here is unfortunately somewhat involved, since we
 need four conformal maps associated with each element, along with some
 quasisymmetric riggings.

 Let $(f^0,g^\infty) \in \mathcal{A}$ be a non-overlapping pair of conformal maps
 $f:\mathbb{D} \rightarrow \mathbb{C}$ and $g:\mathbb{D}^*
 \rightarrow \mathbb{C}$ with quasiconformal extensions. This is a representation of a rigged annulus in the case that
 $(f^0,g^\infty) \in \mathcal{A}^o$.  These
 two maps $f^0$ and $g^\infty$ each have `complementary' conformal
 mapping functions, i.e. the conformal maps onto the complement of
 their image.  We will denote these complementary maps by
 $f^\infty:\mathbb{D}^* \rightarrow \cbar \setminus {\overline{f^0(\mathbb{D})}}$ and
 $g^0:  \mathbb{D} \rightarrow \cbar \setminus{\overline{g^\infty(\mathbb{D}^*)}}$
 respectively. We normalize these maps by requiring $f^\infty(\infty) = \infty, (f^\infty)'(\infty)>0, g^0(0) = 0$ and $(g^0)'(0) >0$.

 The element $(f^0,g^\infty)$ also
 has two quasisymmetric mappings corresponding to $0$ and $\infty$
 namely $\phi^0={f^\infty}^{-1} \circ f^0$ and
 $\phi^\infty={g^\infty}^{-1} \circ g^0$.  (Conventions: the upper indices \
 of $f$ and $g$ always distinguish whether the
 mapping is at zero or infinity.  For
 the quasisymmetries, the inverse map on the left is always the one
 defined at $\infty$).

 \begin{theorem}[multiplication in $\mathcal{A}^0$]
 \label{th:annulimultiplication}
  Let $(f_1^0,g^\infty_1) \in \mathcal{A}^0$ and $(f_2^0,g^\infty_2) \in \mathcal{A}^0$.
  Define
  the two quasisymmetries $\phi_1^0 = {f_1^\infty}^{-1} \circ f_1^0$
  and $\phi_2^\infty = {g_2^\infty}^{-1} \circ g_2^0$.

  Let $(F,G)$ be the conformal welding pair such that
  \[  \phi_1^0 \circ \phi_2^\infty={G}^{-1} \circ F,  \]
  where $F$ and $G$ are normalized by $F(0) = 0, G(\infty) = (\infty)$, $G'(\infty) = (f_1^{\infty}) ' (\infty)$.

  The product of the annuli is then given by
  \begin{equation} \label{eq:annulimultiplication}
   (f_1^0,g_1^\infty) \cdot (f_2^0,g_2^\infty) = \left( F
   \circ {g_2^0}^{-1} \circ f_2^0 \:,\: G \circ {f_1^\infty}^{-1}
   \circ g_1^\infty \right).
  \end{equation}
 \end{theorem}
 \begin{proof}
 Composition preserves quasisymmetries \cite{Lehto}, so
 Theorem \ref{th:conformalwelding} guarantees the existence of the
 welding pair $(F,G)$.
 According to Definition \ref{de:border_multiplication} we sew
 together $(f^0_1, g^\infty_1)$ and $(f^0_2,g^\infty_2)$
 as follows.  Remove
 $f^0_1(\mathbb{D})$ from the first sphere, and
 $g^\infty_2(\mathbb{D^*})$ from the second sphere.  Join the two remaining domains, identifying points on the boundaries via the map
 ${g^\infty_2} \circ {f_1^0}^{-1}$.  Denote the new sphere by
 \[ \mathbb{S}=\overline{\mathbb{C}}\backslash f^0_1(\mathbb{D}) \;
 \#_{{g^\infty_2} \circ {f_1^0}^{-1}} \; \overline{\mathbb{C}}\backslash
 g^\infty_2(\mathbb{D}^*).  \]
where the complex structure on $\mathbb{S}$ is given by Theorem
\ref{th:sewing}.   Figure (\ref{sewing}) below may help visualize
the rest of the proof.

 \begin{figure}[ht]
\begin{center}
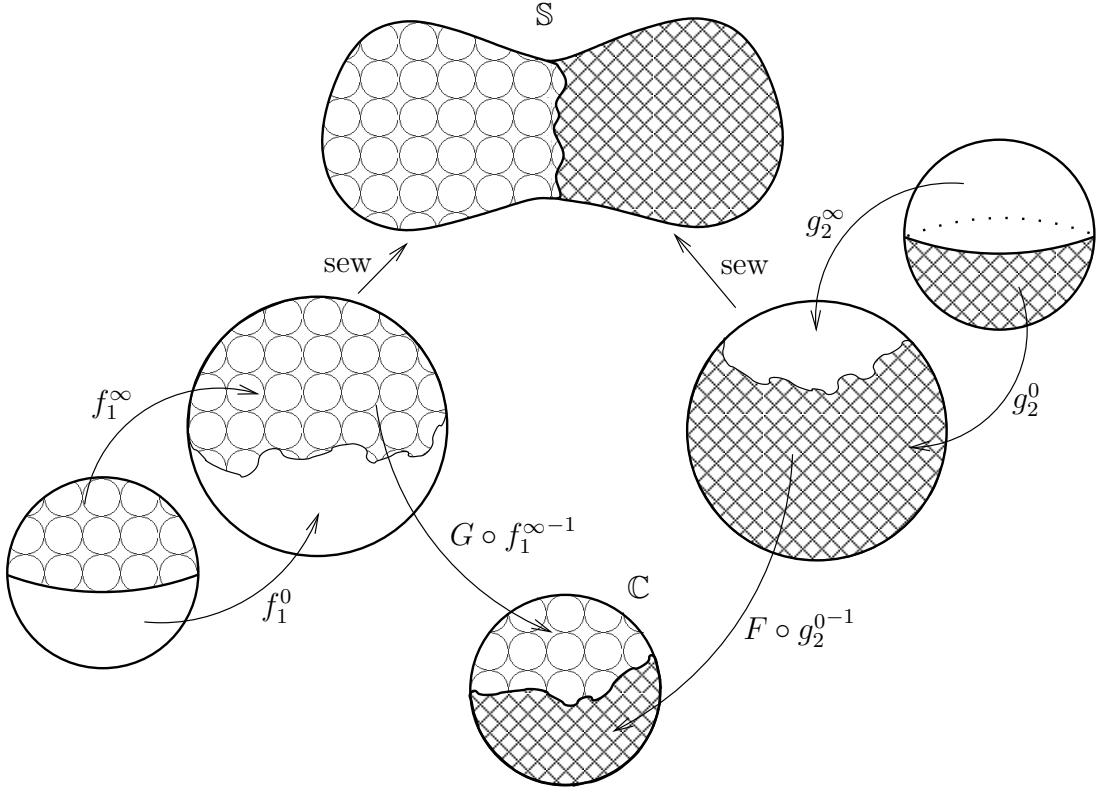
\caption{Sewing spheres using quasisymmetric boundary
identification.} \label{sewing}
\end{center}
\end{figure}

 The problem now is that this new sphere
 is an abstract object, and we need to convert it to a standard
 sphere, and keep track of what happens to the remaining data
 $g^\infty_1$ and $f_2^0$.  To do this, we first give an alternate
 representation of $\mathbb{S}$ as a join of $\mathbb{D^*}$
 to $\mathbb{D}$ (considered as subsets of the spheres corresponding to
 $(f_1^0,g^\infty_1)$ and
 $(f^0_2,g^\infty_2)$ respectively.

 These are sewn via $\phi_1^0
 \circ \phi_2^\infty= {f_1^\infty}^{-1} \circ f_1^0 \circ
 {g_2^\infty}^{-1} \circ g_2^0$.
 That is,
 \[ \mathbb{S}' \equiv \mathbb{D} \; \#_{\phi_1^0 \circ
 \phi_2^\infty} \; \mathbb{D}^*  \]
 which is conformally equivalent to $\mathbb{S}$ via the map $\mathbb{S} \to \mathbb{S}'$ given by
 $$
 z  \mapsto
 \begin{cases} f_1^{\infty}(z), & \text{if } z \in \mathbb{D}^* \\
 g_2^0(z), &  \text{if } z \in \mathbb{D}.
 \end{cases}
$$

 The final step is to represent $\mathbb{S}'$ as two complementary
 quasidisks on the standard sphere joined by the identity map.

 By Theorem \ref{th:conformalwelding} there is a unique pair
 of complementary mappings $F$ and $G$
 associated with the
 quasisymmetry $\phi_1^0 \circ \phi_2^\infty$ such that $\phi_1^0 \circ \phi_2^\infty =
 {G}^{-1} \circ F$ and $G'(\infty)={f_1^\infty}'(\infty)$.
 The map which is equal
 to $F$ on $\mathbb{D}$ and $G$ on $\mathbb{D}^*$ extends
 continuously to a biholomorphic map from $\mathbb{S}'$ into
 $\overline{\mathbb{C}}$.

 Thus we have that the map which is equal to $F \circ
 {g_2^0}^{-1}$ on $\overline{\mathbb{C}} \backslash g^\infty_2(\mathbb{D}^*)$
 and $G \circ {f^\infty_1}^{-1}$ on $\overline{\mathbb{C}}
 \backslash f_1^0(\mathbb{D})$ extends continuously to a
 biholomorphism of $\mathbb{S}$ onto $\overline{\mathbb{C}}$.
 The new data is obtained by composing this biholomorphism with the riggings
 $g_1^\infty$ and $f_2^0$: at $\infty$,
 we have $G \circ {f^\infty_1}^{-1} \circ g_1^\infty$, and
 at $0$ we have $F \circ {g_2^0}^{-1} \circ f_2^0$.
 \end{proof}
 \begin{remark}  The multiplication above extends to
 $\mathcal{A}$ without difficulty.  However, this extended multiplication
 cannot be considered a consequence of Definition \ref{de:border_multiplication}.  The
 interpretation of the multiplication in the non-overlapping mapping
 picture does not change.
 \end{remark}

\end{subsection}
\begin{subsection}{Two natural sub-semigroups of $\mathcal{A}$}
 $\mathcal{A}$ possesses two natural sub-semigroups
 \cite{Neretin_holomorphic,Segal}.
 in the case that the boundary parametrizations are diffeomorphisms.
 Allowing the case of degenerate annuli simplifies matters.
 We include an exposition of the ideas here, both to verify that the picture
 holds in the case of quasisymmetric and for the convenience of the reader.

 \begin{definition}
  Let $\mathcal{E}$ denote the subset of $\mathcal{A}$
  consisting of elements of the form $(f_1^0,\mbox{Id})$.  Let $\mathcal{E}^o = \mathcal{E} \cap \mathcal{A}^o$.
 \end{definition}
 Note that in particular, this implies that $f_1^0 : \cdisk \subset \disk$ is a bounded
 univalent function with quasiconformal extension.

 \begin{proposition} \label{pr:E_multiplication}
   $\mathcal{E}$ is a submonoid of $\mathcal{A}$ and $\mathcal{E}^0$ is a subsemigroup of $\mathcal{A}^0$.  In both cases the multiplication is given by
 \[  (f_1^0,\text{Id}) \cdot (f_2^0,\text{Id}) = (f_1^0 \circ f_2^0, \text{Id}).  \]
 \end{proposition}
 \begin{proof}
 We need to establish that $\mathcal{E}$ is indeed closed under
 multiplication.  If $(f_1^0,\mbox{Id}),(f_2^0,\mbox{Id}) \in
 \mathcal{E}$ then the maps $G$ and $F$ of Theorem
 \ref{th:annulimultiplication} satisfy ${G}^{-1} \circ
 F={f_1^\infty}^{-1} \circ f_1^0$ since
 $g_i^\infty=g_i^0=\mbox{id}$ for $i=1,2$.  Thus $G =
 f_1^\infty$ and $F=f_1^0$, so by the definition of multiplication
 \ref{eq:annulimultiplication} it follows that
 \begin{equation} \label{eq:Emultiplication}
  (f_1^0,\mbox{Id}) \cdot (f_2^0,\mbox{Id}) = (f_1^0 \circ
   f_2^0,\mbox{Id}).
 \end{equation}
 It remains to show that $\mathcal{E}^o$ is closed under multiplication.
 This follows from the observations that if $f_1^0(\overline{\mathbb{D}})$
 does not intersect $S^1$, then neither does
 $f_1^0 \circ f_2^0(\overline{\mathbb{D}})$, and that a composition of
 quasiconformal maps is also a quasiconformal map.
 \end{proof}

 The set of quasisymmetries of the circle $QS(S^1)$ is a group under
 composition.  One can regard $\mbox{QS}(S^1)$ as consisting of ``degenerate
 annuli'' corresponding to welding pairs.  That is, if $\phi:S^1
 \rightarrow S^1$ is a quasisymmetry we can view the corresponding
 welding pair $(f^0,g^\infty)$ such that ${g^\infty}^{-1} \circ f^0 = \phi$
 as an element of $\mathcal{A}$.
 \begin{definition} Let $\mathcal{G}\subset \mathcal{A}$ denote the set of
 pairs $(f,g) \in \mathcal{A}$ such that $f(\partial \mathbb{D}) = g(\partial \mathbb{D}^*)$ as sets.
 \end{definition}
 $\mathcal{G}$ is a group and multiplication corresponds to composition of
 quasisymmetries, as the following proposition shows.
 \begin{proposition} \label{pr:G_multiplication}
  $\mathcal{G}$ is a group, and is isomorphic to $\mbox{QS}(S^1)$ via the map
  \begin{align*}
   \rho:\mathcal{G} & \rightarrow   \mbox{QS}(S^1) \\
    (f,g)& \mapsto  \left.g^{-1} \circ f \right|_{S^1}.
  \end{align*}
 \end{proposition}
 \begin{proof}
 Let
 $(f_1^0,g_1^\infty)$ and $(f_2^0,g_2^\infty)$ be two such welding
 pairs corresponding to $\phi_1$ and $\phi_2$ respectively.
 Thus we have that $g_i^0=f_i^0$ and
 $g_i^\infty=f_i^\infty$ for $i=1,2$, and furthermore that
 $\phi_1^0=\phi$ and $\phi_2^\infty=\phi_2$ in Definition
 \ref{th:annulimultiplication}.  Thus by equation
 \ref{eq:annulimultiplication} it follows that
 \[ (f_1^0,f_1^\infty) \cdot(f_2^0,f_2^\infty) = (F,G)  \]
 where ${G}^{-1} \circ F = \phi_1 \circ \phi_2$, which
 shows that $\rho$ is a homomorphism.  If $\rho(f,g)=g^{-1} \circ
 f = \mbox{Id}$, then $f=g=\mbox{Id}$, so $\rho$ is injective.
 The fact that $\rho$ is surjective follows directly from Theorem
 \ref{th:conformalwelding} with $a=1$.
 \end{proof}
 \begin{remark}  In other words, if $\phi_1={g_1^{\infty}}^{-1}
 \circ f_1^0$ and $\phi_2={g_2^{\infty}}^{-1} \circ f_2^0$ then
 \[ \rho\left((f_1^0,g_1^\infty) \cdot (f_2^0,g_2^\infty)\right)=\phi_1 \circ
 \phi_2.  \]
 \end{remark}
\end{subsection}
\begin{subsection}{A complex structure on the semigroup of rigged
annuli} \label{se:complex_structure_definition} Let $\cpunc$ denote
the twice-punctured sphere $\mathbb{C} \setminus \{0 \}$. We will
now define a natural complex structure on $\mathcal{A}^0$.  This is
inherited from a set of non-overlapping maps, which we now define.

\begin{definition}
Let $\Oqc(\cpunc) = \{(f,g)\}$ where $(f,g)$ is a pair of
non-overlapping mappings satisfying conditions $(1)$, $(2)$ and
$(3)$ in Definition \ref{de:Ao} and $g(\infty) = \infty$.
\end{definition}

The space $\mathcal{A}^o$ can be identified with a quotient of
$\Oqc(\cpunc)$ as follows. Define a $\cpunc$ action on
$\Oqc(\cpunc)$ by
$$ a \cdot (f,g) \mapsto (af,ag)  $$
for $a \in \mathbb{C}^*$.  This is the action of the automorphism
group $\mbox{Aut}(\cpunc)=\{ z \mapsto az : a \in \cpunc \}$ by
composition on the left. The quotient space will be denoted
$$
\Oqc(\cpunc)/\mbox{Aut}(\mathbb{C}^*)
$$
with elements denoted $[f,g]$. Each element of
 $\Oqc(\cpunc)/\mbox{Aut}(\mathbb{C}^*)$ has a unique representative in
 $\mathcal{A}^o$.
Thus
\begin{align}
\label{Ao_Oqc_iso}
I : \mathcal{A}^o & \to \Oqc(\cpunc)/ \Aut(\cpunc) \\
(f,g) &\mapsto [f,g] \nonumber
\end{align}
is a bijection.

\begin{remark}
 Although $\Aut(\cpunc)$ is
 the same as $\cpunc$ as a set, we will keep distinct notation in
 the quotient as
 they have different roles.
 \end{remark}

 We can endow $\Oqc(\cpunc)/\mbox{Aut}(\mathbb{C}^*)$ with a complex structure
 in two distinct ways.  We describe one of these ways now.
 First, we need to define two function
 spaces.
 \begin{definition} \label{de:Oqcdefinition}
  Let $\Oqc$ denote the set of $f:\mathbb{D} \rightarrow
 \mathbb{C}$ satisfying $f(0)=0$ which are holomorphic, one-to-one
 and possess a quasiconformal extension to $\mathbb{C}$.
 \end{definition}
 \begin{definition}  Let
 \[  A^1_\infty(\mathbb{D}) = \{ v(z): \mathbb{D} \rightarrow
 \mathbb{C} \;|\; v \text{ holomorphic}, \ ||v||_{1,\infty}=
 \sup_{z \in \mathbb{D}}
 (1- |z|^2) | v(z)| < \infty \}.  \]
 \end{definition}
 Note that $A^1_\infty(\mathbb{D})$ is a Banach space. Let
 $A^1_\infty(\mathbb{D})\oplus \mathbb{C}$ denote the Banach space
 with the direct sum norm
 $||(\phi,c)||=||\phi||_{1,\infty} + |c|.$

 The function space $\Oqc$ has a complex structure derived from that
 of $A^1_\infty(\mathbb{D})
 \oplus \mathbb{C}$.  Define
 \begin{equation}  \label{eq:scriptA_definition}
  \Psi(f)=\frac{f''}{f'}.
 \end{equation}
 The image of $\Oqc$ under the map
 \begin{equation} \label{eq:chi_definition}
  \chi(f)= \left( \Psi(f),f'(0) \right)
 \end{equation}
 is an open subset of $A^1_\infty(\disk) \oplus \mathbb{C}$ and thus
 $\chi$ induces a complex structure on $\Oqc$ \cite[Theorem
 3.1]{RadSchip_nonoverlapping}.

 $\Oqc(\cpunc)/\mbox{Aut}(\mathbb{C}^*)$ also inherits a complex structure
 from $\Oqc(\cpunc)$.  It was shown in
 \cite{RadSchip_nonoverlapping} that the set of non-overlapping
 quasiconformally extendible conformal maps into a Riemann surface with $n$
 distinguished points possesses
 a complex structure, locally modelled on an $n$-fold product of $\Oqc$.
 In particular, $\Oqc(\cpunc)$ can be given a complex structure in
 this way.
 In fact, $\Oqc(\cpunc)$ is mapped bijectively onto an open subset
 of $\Oqc \times \Oqc$ via the map
 \begin{eqnarray} \label{eq:thetadefinition}
  \theta: \Oqc(\cpunc) & \rightarrow & \Oqc \times \Oqc \\
   (f,g)& \mapsto & (f,S(g)) \nonumber
 \end{eqnarray}
 where
 \begin{equation} \label{eq:S_definition}
  S(g)(z)=1/g(1/z).
 \end{equation}
 \begin{theorem} \label{th:Oqccstar_in_Oqcsquared}
  $\theta$ is an injective map onto an open subset of $\Oqc \times
  \Oqc$.  Thus $\Oqc(\cpunc)$ inherits a complex structure from
  $\Oqc \times \Oqc$.
 \end{theorem}
 \begin{proof}
  It is obvious that $\theta$ is injective.  We show that $\theta$
  is open.

   Define $\iota(z)=1/z$.  Fix $(f_0,g_0) \in \Oqc(\cpunc)$.
 Let $B^0$ and $B^\infty$ be simply connected open sets containing
 $\overline{f_0(\disk)}$ and $\overline{g_0(\disk^*)}$ such that
 $B^0 \cap B^\infty$ is empty. Choose $\zeta^0:B^0 \rightarrow \mathbb{C}$ and
 $\zeta^\infty: B^\infty \rightarrow \cbar \setminus \{0\}$ to be
 one-to-one holomorphic maps taking $0$ to $0$ and $\infty$ to $0$ respectively.

 By \cite[Corollary 3.5]{RadSchip_nonoverlapping} there is
 an open neighbourhood $U_0$ of $\zeta^0 \circ f_0 \in \Oqc$ such that for
 all $\psi_0 \in U_0$, $\overline{\psi_0(\disk)} \subset \zeta^0(B^0)$. Similarly,
 there is a neighbourhood $U_{\infty}$ of $\zeta^\infty \circ g_0 \circ \iota$
 such that for
 all $\psi_\infty \in U_\infty$, $\overline{\psi_\infty(\mathbb{D})} \subset
 \zeta^\infty(B^\infty)$.  Thus $\theta^{-1}(U_0 \times U_\infty)
 \subset \Oqc(\cpunc)$.  This proves that $\theta$ is open.
 \end{proof}

 Thus we have a global map of $\Oqc(\cpunc)$ into an open susbset
 of a  Banach space given by
 \begin{eqnarray} \label{eq:nice_chart}
   B: \Oqc(\cpunc) & \rightarrow & A^1_\infty(\disk) \oplus \mathbb{C}^* \oplus A^1_\infty(\disk) \oplus \mathbb{C}^*
   \\ \nonumber
  (f,g) & \mapsto & (\chi(f), \chi(S(g))) = \left(
  \mathcal{A}(f),f'(0), \mathcal{A}(S(g)),g'(\infty) \right).
 \end{eqnarray}

 We will show how this complex structure passes down to the quotient
 $\Oqc(\cpunc)/\mbox{Aut}(\mathbb{C}^*)$.  Denote
 for $a \in \mathbb{C}^*$,
 \[  \Oqc(\cpunc)_a =\{ (f,g) \in \Oqc(\cpunc) : g'(\infty)=a \}.
 \]
 For any $a$, $\Oqc(\cpunc)/\mbox{Aut}(\cpunc)$ can be identified with $\Oqc(\cpunc)_a$.
 Of course $\Oqc(\cpunc)_1=\mathcal{A}^o$, and
 $\Oqc(\cpunc)_a$ are just cosets of $\mathcal{A}^0$ under the
 group action of $\mbox{Aut}(\cpunc) \cong \mathbb{C}^*$ on $\Oqc(\cpunc)$.
 Observe that
 \begin{proposition} \label{pr:autc_action}
  The action $a \cdot (f,g)=(af,ag)$ of $\mbox{Aut}(\mathbb{C}^*)$ on
  $\Oqc(\cpunc)$ is holomorphic.
 \end{proposition}
 \begin{proof}
  Since $\Oqc(\cpunc)$ is locally modelled on $\Oqc \times \Oqc$,
  by Hartog's theorem on separate holomorphicity (see \cite{Mujica} for
  a version in infinite dimensions) it is enough to check that $f
  \mapsto a f$ and $S(g) \mapsto
  S(g)/a$ are holomorphic maps of $\Oqc$.  Since $T(z)=bz$ is a holomorphic
  map on $\cbar$ this follows immediately from
  \cite[Lemma 3.10]{RadSchip_nonoverlapping}.
 \end{proof}
 Define the map
 \begin{eqnarray} \label{eq:H_definition}
  H: \Oqc(\cpunc) & \rightarrow & \mathbb{C}^* \\ \nonumber
  (f,g) & \mapsto & g'(\infty).
 \end{eqnarray}
 \begin{theorem} \label{th:Oqc_structure_quotient}
  $H$ is holomorphic, and possesses a global holomorphic section $s$.
 \end{theorem}
 \begin{proof}
  The map $S(g) \mapsto S(g)'(0)=g'(\infty)$ is just $\chi$ followed by a
  projection onto the second component of $A^1_\infty(\disk) \oplus
  \mathbb{C}$, which is clearly holomorphic since $\chi$ is.
  Thus $H$ is holomorphic.

  Fix $(f_0,g_0) \in \Oqc(\cpunc)$.  $H$ has a global section
  through $(f_0,g_0)$ given by
  \begin{align*}
  s : \mathbb{C}^* &\to \Oqc \\
    a &\mapsto \frac{a}{g_0'(\infty)} \cdot (f_0,g_0).
  \end{align*}

  It is easy to check that $H \circ s(a)=a$ is the identity for any
  $a \in \mathbb{C}$.  Now
  \[  \mathcal{A} \left( \frac{a}{g_0'(\infty)} f_0 \right) =
    \mathcal{A}(f_0) \]
  and similarly for $S(g_0)$.  Since
  \[ a \mapsto B(s(a))= \left( \mathcal{A}(f_0), a \frac{f_0'(\infty)}
  {g_0'(\infty)},\mathcal{A}(S(g_0)),a \right) \]
  is holomorphic, it follows that
   $s$ is holomorphic.
 \end{proof}
 \begin{corollary} \label{co:normalized_Oqc_submanifold}
  For every $a \in \mathbb{C}^*$, $\Oqc(\cpunc)_a$
  is a complex submanifold of $\Oqc(\cpunc)$.
 \end{corollary}
 \begin{proof} This follows from Proposition \ref{pr:autc_action}
 and Theorem \ref{th:Oqc_structure_quotient}.
  See \cite[Lemmas 2.15 and 2.16]{RadSchip_fiber} and references therein.
 \end{proof}

We can transfer the complex structure from $\Oqc(\cpunc)_a$ to
$\Oqc(\cpunc)/\mbox{Aut}(\mathbb{C}^*)$ as follows.
 The map
\begin{align*}
r_a : \Oqc(\cpunc)/\Aut(\cpunc) & \to \Oqc(\cpunc)_a \\
[f,g] & \mapsto \frac{a}{g'(\infty)}(f,g)
\end{align*}
is a bijection.
 By the Corollary, $\Oqc(\cpunc)/\mbox{Aut}(\mathbb{C}^*)$ inherits a complex structure from
 $\Oqc(\cpunc)$ via $r_a$ in which $r_a$ is automatically a biholomorphism.

Combing this fact with Theorem \ref{th:mba_Ao_iso} and the bijection $I$ in (\ref{Ao_Oqc_iso}) we obtain the following important result.
\begin{theorem} \label{th:complexstructure_via_Oqc}
The two moduli spaces of rigged annuli, $\mba$ and  $\mathcal{A}^0$ inherit complex structures from $\Oqc(\cpunc)$ via the bijections
$$
\mba \stackrel{R}\longrightarrow \mathcal{A}^o \stackrel{I}\longrightarrow \Oqc(\cpunc)/\Aut(\cpunc) \stackrel{r_a}\longrightarrow \Oqc(\cpunc)_a \subset \Oqc(\cpunc) .
$$
\end{theorem}

\end{subsection}
\end{section}
\begin{section}{The relation between the moduli space of rigged annuli and the
Teichm\"uller space of the annulus} \label{se:relation_with_Teich}
\begin{subsection}{Teichm\"uller space and modular groups}

 \label{ssec:teich_def}

Let $\riem$ be a bordered Riemann surface of of type $(g,n)$.
 Consider the set of triples
 $\{(\riem,f_1,\riem_{1})\}$ where $\riem$ is a fixed Riemann
 surface, $\riem_{1}$ is another Riemann surface and $f_1:\riem
 \rightarrow \riem_{1}$ is a quasiconformal map.
 We say that $(\riem,f_1,\riem_{1}) \sim (\riem,f_2,\riem_{2})$
 if there exists a biholomorphism $\sigma:\riem_{1} \rightarrow
 \riem_{2}$ such that $f_2^{-1} \circ \sigma \circ f_1$ is
 homotopic to the identity ``rel boundary''.  ``Rel boundary''
 means that
 the restriction of $f_2^{-1} \circ \sigma \circ f_1$ to the boundary
 is the identity throughout the homotopy.
 \begin{definition}
 The Teichm\"uller space of $\riem$ is
 \[  T(\riem) = \{ (\riem,f_1,\riem_{1})\}/\sim.  \]
 We denote the equivalence classes by $[\riem,f_1,\riem_{1}]$.
 \end{definition}

 It is well known that $T(\riem)$ is a complex Banach manifold with complex structure
 compatible with the space $L^\infty_{-1,1}(\riem)_1$ of Beltrami differentials
 via the fundamental projection $\Phi: L^\infty_{-1,1}(\riem)_1 \to T(\riem)$.
 Here $L^\infty_{-1,1}(\riem)$ denotes the set of $(-1,1)$
 differentials $\mu d\bar{z}/dz$ with bounded essential supremum, and
 $L^\infty_{-1,1}(\riem)_1$ denotes the unit ball in
 $L^\infty_{-1,1}(\riem)$.
 The fundamental projection takes Beltrami differentials $\mu d\bar{z}/dz$ to
 quasiconformal maps whose dilatation is $\mu d\bar{z}/dz$.

As in \cite[section 2.1]{RadSchip05} we introduce a certain subgroup
of the mapping class group.  Proofs and details can be found there.
The pure mapping class group of $\riem$ is the group of homotopy
classes of quasiconformal self-mappings of $\riem$ which preserve
the ordering of the boundary components. Let $\pmcgi[\riem]$ be the
subgroup of the mapping class group consisting of equivalence
classes of mappings that are the identity on the boundary  $\partial
\riem$. The group $\pmcgi[\riem]$ is finitely generated by Dehn
twists.

 The mapping class group acts on $T(\riem)$ by $[\rho] \cdot
[\riem,f,\riem_1] = [\riem, f \circ \rho, \riem_1]$.
\begin{proposition}[{\cite[Lemmas 5.1 and 5.2]{RadSchip05}}]
\label{DBaction} The group $\pmcgi[\riem]$ acts properly
discontinuously and fixed-point freely by biholomorphisms on
$T(\riem)$.
\end{proposition}

 In the case of an annulus $A$, $\pmcgi[A] \simeq \mathbb{Z}$. See
 \cite[Proposition 2.1]{RadSchip05} for the general case and references.
 Note that a representative of an element in $\pmcgi[A]$ is a quasiconformal map
 $ g:A \rightarrow A$  such that $ \left. g \right|_{\partial A} =
  \text{Id}$. We subsequently identify $\pmcgi[A]$ with $\mathbb{Z}$.

The quotient of $T(A)$ by $\pmcgi[A]$ is in one-to-one
correspondence to $\mba$. To exhibit the bijection, fix a rigging
$\tau$ of the base annulus $A$. Define
\begin{align}
F : T(A) /  \pmcgi[A] & \to \mba  \\
[A,f,A_1] & \mapsto [A_1, f \circ \tau] \nonumber .
\end{align}
Note that $f \circ \tau$ is a quasisymmetric rigging because the boundary values of the quasiconformal map $f$ are necessarily quasisymmetric.
\begin{theorem}
\label{th:TA_is_M} The map $F$ is a bijection and hence $\mba$
inherits a complex Banach manifold structure from $T(A)$.
\end{theorem}
\begin{proof}
This is the special case of \cite[Theorem 5.3]{RadSchip05} with
$g=1$,$n^- = 1$ and $n^+ = 1$.
\end{proof}
\begin{remark} The non-trivial work in Theorem \ref{th:TA_is_M} is in showing $F$
is onto. It requires the existence of quasiconformal maps between
annuli with specified quasisymmetric boundary values, the proof of
which ultimately relies on the extended lambda-lamma.
\end{remark}

\begin{remark}
The above theorem is in fact not needed in the logical development
 of this paper because we exhibit three other bijections
 $K \circ P^{-1}$, $R$ and $I$ which can be combined to give $F$ (see diagram
 (\ref{teepee})).  It is included because the explicit map can be
 easily written.
\end{remark}

\end{subsection}
\begin{subsection}{Identification of rigged annuli and
$T(A)/\mathbb{Z}$}
  To describe the relation between $T(A)$ and $\mathcal{A}^0$
  we will need the following result.

 \begin{lemma} \label{le:when_boundary_equal}
 Let $[A,h_i, A_i] \in T(A)$ for $i=1,2$. Then $[A,h_1, A_1]$  and
  $[A,h_2, A_2]$ are equivalent in the quotient $T(A) / \mathbb{Z}$ if and only
  if there exists a biholomorphism $\sigma : A_1 \rightarrow A_2$ such that
 $h_2=\sigma \circ h_1$ on $\partial A$.
 \end{lemma}
 \begin{proof}  $h_2=\sigma \circ h_1$ on $\partial A$ if and only
 if $h_2^{-1} \circ \sigma \circ h_1 = \text{Id}$ on $\partial A$ if and
 only if $h_2^{-1} \circ \sigma \circ h_1 \in \pmcgi[A]$. We previously
 observed that $\pmcgi[A] \cong \mathbb{Z}$.
 \end{proof}

 To describe the quotient map from $T(A)$ to
 $\Oqc(\cpunc)/\Aut(\cpunc)$ we will choose the base surface $A \subset \cpunc$ and choose a canonical
 representation of elements of $T(A)$.

Given any doubly-connected bordered Riemann surface $A^*$ and any
rigging $\tau^*$, the caps can be sewn on to obtain $\riem^P$ as in
equation (\ref{riemP}) together with the biholomorphically extended
rigging $\tilde{\tau}^*$ as in equation
(\ref{eq:tauextension_definition}). Since $\riem^P$ has two
punctures and genus zero, we may choose a biholomorphism $\sigma :
\riem^P \to \cpunc$. Let $A = \sigma(\overline{A^*}), C_0 =
\sigma(\cdisk_0)$ and $C_{\infty} = \sigma(\pcdiski)$. The map $\tau
= \sigma \circ \tau^*$ is a rigging of $A$ and $\tilde{\tau} =
\sigma \circ \tilde{\tau}^*$ is its biholomorphic extension to the
caps $C_1$ and $C_2$.

Thus without loss of generality we may choose our base annulus to
be an $A \subset \cpunc$ which is bounded by quasicircles and give
the base annulus a rigging $\tau = (\tau_0, \tau_{\infty})$ that
extend to biholomorphisms $\tilde{\tau}_0: \disk_0 \to C_0$ and
$\tilde{\tau}_{\infty}: \disk^*_{\infty} \to C_{\infty}$. We
henceforth work with such a base surface and rigging.
 \begin{definition}  A standard base is $(A,\tau)$
 where $A$ is a doubly-connected region in $\mathbb{C}^*$ bounded by
 quasicircles, and $\tau=(\tau_0,\tau_\infty)$ are quasisymmetric
 riggings which extend to biholomorphisms $\tilde{\tau}_0: \disk_0 \to C_0$
 and $\tilde{\tau}_{\infty}: \disk^*_{\infty} \to C_{\infty}$.
 \end{definition}
 \begin{definition} \label{de:canonical_rep}  Let $(A,\tau)$ be a
 standard base.
 We call a representative $(A,h,A')$ of an element in $T(A)$, a {\it
  canonical representative} if $A'$ is a doubly-connected subset of $\cpunc$ whose
  boundaries are quasicircles, and $h$ is the restriction to $A$ of
  a quasiconformal map $\tilde{h}:\cpunc \rightarrow \cpunc$ which is
  a biholomorphism from the complement of $\overline{A}$ to the
  complement of $\overline{A'}$.
 \end{definition}
 \begin{proposition}  \label{pr:representative_tba}
  Let $(A,\tau)$ be a standard base.  Every element of
  $T(A)$  has a canonical representative.
 \end{proposition}
 \begin{proof} Choose an arbitrary element $[A,h_1, A_1] \in T(A)$.
 Sew caps onto $A$ using $\tau = (\tau_0, \tau_{\infty})$. Then
 $$
 S=A^* \#_\tau (\cdisk_0 \sqcup \pcdiski)
 $$
 is biholomorphically equivalent to $\cpunc$ via the continuous
 extension of
 $$
 F(x) =
 \begin{cases}
 \tilde{\tau}_0(x), & \text{for } x \in \disk_0 \\
 x, & \text{for } x \in A \\
 \tilde{\tau}_{\infty}(x) , & \text{for } x \in \disk^*_{\infty}
 \end{cases}
 $$
 where $\tilde{\tau}_0$ and $\tilde{\tau}_\infty$ are the
 biholomorphic extensions of $\tau_0$ and $\tau_\infty$.

 Now sew caps onto $A_1$ via the rigging $h_1 \circ \tau$ to obtain the Riemann surface
 $$
   S_1 = A_1 \#_{h_1 \circ \tau} (\cdisk_0 \sqcup \pcdiski)
 $$
  which is biholomorphic to $\cpunc$ via some map
  $\sigma : S_1 \rightarrow \cpunc$. Setting $h= F^{-1} \circ \sigma \circ h_1 = \sigma \circ h_1$ and $A'=\sigma(A_1)$ we have $[A,h,A'] \in T(A)$. Here we consider $A_1 \subset S_1$.

 Consider the diagram
 $$
 \xymatrix@+10pt{S \ar[r]^{h'_1} \ar[d]^F & S_1 \ar[d]^{\sigma} \\
 \cpunc \ar[r]^{\tilde{h}} & \cpunc}
 $$
 where $h'_1$ is defined by
 \begin{equation}
\label{hprime} h'_1(x) =
\begin{cases}
h_1(x) & \text{for }  x \in A \\
x & \text{for } x \in \pcdisk \sqcup \pcdiski .
\end{cases}
\end{equation}
and $\tilde{h} = \sigma \circ h'_1 \circ F^{-1}$.
 That is $\tilde{h}: \cpunc \to \cpunc$ is the quasiconformal extension of $h$ given by
 $$  \tilde{h}(z) =
 \begin{cases}
  h = \sigma \circ
  h_1(z),
  & \text{for } z \in A \\
  \sigma \circ \tilde{\tau}_0^{-1}(z),  & \text{for } z \in \tau_0(\disk) \\
  \sigma \circ \tilde{\tau}_\infty^{-1}(z), & \text{for } z \in \tau_\infty(\disk^*).
  \end{cases}
  $$
  From the definition of the sewing operation, $h'$ is continuous, and
hence it is  quasiconformal by the removability of quasiarcs for
quasiconformal mappings \cite[V.3]{LV} (see also \cite[Section
5.3]{RadSchip05}). Moreover, the restriction of $\tilde{h}$ to the
caps is biholomorphic and so $\tilde{h} \circ \tilde{\tau} \in
\Oqc(\cpunc)$.
  Clearly $\tilde{h}$ is conformal on the complement of $A$ and satisfies $\tilde{h}(0)=0$ and $\tilde{h}(\infty) = \infty)$.

 \end{proof}

 \begin{definition} \label{de:K_definition}
  Let $(A,\tau)$ be a standard base and let
  \[  K: T(A) \rightarrow \Oqc(\cpunc)/\Aut(\cpunc) \]
  be defined by
  \[  K([A,h,A'])= [\tilde{h} \circ \tau_0, \tilde{h} \circ
  \tau_\infty]  \]
  where $(A,h,A')$ is a canonical representative  and $\tilde{h}$ is the corresponding
  extension as in Proposition
  \ref{pr:representative_tba}.
 \end{definition}
 It must be shown that $K$ is well-defined.  Assume that
 $(A,h_1,A_1)$ and $(A,h_2,A_2)$ are canonical representatives
 which are equivalent in $T(A)$.  Then there exists a biholomorphism
 $\sigma:A_1 \rightarrow A_2$ such that $h_2^{-1} \circ \sigma \circ
 h_1$ is homotopic to the identity rel boundary.  Now $\sigma$ can
 be extended to a M\"obius transformation as follows. Define
 \[  \tilde{\sigma}(z) = \left\{ \begin{array}{ll} \sigma(z) & z \in
   A_1 \\ \tilde{h}_2 \circ \tilde{h}_1^{-1}(z) & z \in
   \tau_0(\disk) \cup \tau_\infty(\disk^*) \end{array} \right.  \]
 Since $h_2^{-1} \circ \sigma \circ h_1 = \text{Id}$ on $\partial A$,
 this has a continuous extension across the join since $h_2^{-1} \circ \sigma
 \circ h_1$
 is the identity on $\partial A$, so this map must thus be
 quasiconformal on $\cbar$.  Since it is holomorphic except on a
 quasicircle, it is in fact holomorphic everywhere \cite[V.3]{LV}.
 Thus it is a
 M\"obius transformation which takes $0$ to $0$ and $\infty$ to
 $\infty$.  So $\tilde{\sigma}(z)=az$, and hence $\sigma \circ h_1=h_2$ on
 $\partial A$.  Thus $a\tilde{h}_1 = \tilde{h}_2$ on the complement of $A$, and so $(a\tilde{h}_1 \circ \tau_0, a \tilde{h}_1
 \circ \tau_\infty) = ( \tilde{h}_2 \circ \tau_0, \tilde{h}_2 \circ
 \tau_\infty)$.  This shows that $K$ is well-defined.

 $K$ is surjective, but fails to be injective.
 Let $P:T(A) \rightarrow T(A)/\mathbb{Z}$ denote the quotient
 map.  The action by $\mathbb{Z}$ is properly discontinuous and
 has local holomorphic sections (see Proposition \ref{DBaction}).  It
 turns out that $K \circ P^{-1}$ is a well-defined map from an open subset of
 $T(A)/\mathbb{Z}$ to $\Oqc(\cpunc)/\Aut(\cpunc)$, which we will later show is a
 biholomorphism onto its image.
 \begin{proposition} \label{pr:K_surjective_injective}
  $K$ is surjective.  $K([A,h_1,A_1])=K([A,h_2,A_2])$  if and
  only if $[A,h_1,A_1]$ and
  $[A,h_2,A_2]$ are equivalent mod $\mathbb{Z}$.  Thus
  \[  K \circ P^{-1}:T(A)/\mathbb{Z} \rightarrow
  \Oqc(\cpunc)/\Aut(\cpunc)  \]
  is a well-defined bijection.
 \end{proposition}
 \begin{proof}
  We first show the injectivity of $K$ up to the
  $\mathbb{Z}$ action.
  Assume that $K([A,h_1,A_1])=K([A,h_2,A_2])$.
  We can assume that $(A,h_1,A_1)$ and $(A,h_2,A_2)$ are
  canonical representatives.
  Thus $[\tilde{h}_1
  \circ \tau_0, \tilde{h}_1 \circ \tau_\infty] =[\tilde{h}_2
  \circ \tau_0, \tilde{h}_2 \circ \tau_\infty]$.  So $a
  \tilde{h}_1 = \tilde{h}_2$ on $\cbar \setminus A$.  Since
  in general $(A, a h, a A')$ is equivalent to $(A, h, A')$ in
  $T(A)$ we can assume that $A_2=A_1$ and $\tilde{h}_2=
  \tilde{h}_1$ on $\cbar \backslash A$. In particular, $h_1$ and
  $h_2$ agree on $\partial A$, so by Lemma
  $\ref{le:when_boundary_equal}$ they are equivalent mod
  $\mathbb{Z}$.

  Conversely, assume that $[A,h_1,A_1]$ and $[A,h_2,A_2]$ are
  equivalent mod $\mathbb{Z}$.   Then the extensions $\tilde{h}_1$ and
  $\tilde{h}_2$ agree on $\cbar \setminus A$ by Lemma
   \ref{le:when_boundary_equal} and so
  $K([A,h_1,A_1])=K([A,h_2,A_2])$.

  Finally, we demonstrate that $K$ is onto.  By \cite[Corollary 4.1]{RadSchip05}  for any $(f,g) \in
  \Oqc(\cpunc)$ there is a quasiconformal extension
  $\tilde{h}:\cbar \rightarrow \cbar$ of $f \circ \tau_0^{-1}$ and
  $g \circ \tau_\infty^{-1}$.  Set $A'= \tilde{h}(A)$ and we
  have that $K([A,h,A'])=(f,g)$.
 \end{proof}

 In proving the surjectivity of $K$, we very nearly defined the
 inverse of $K \circ P^{-1}$.
 \begin{proposition} \label{pr:L_is_KofPinverseinverse}
 The inverse of $K \circ P^{-1}$ is
 \begin{align*} \label{eq:L_definition}
  L: \Oqc(\cpunc)/\Aut(\cpunc) & \rightarrow  T(A)/\mathbb{Z}
  \\ \nonumber
  [f,g] & \mapsto  (A,\left.\tilde{h} \right|_A, \tilde{h}(A))
 \end{align*}
 where $\tilde{h}$ is a quasiconformal extension to $\cpunc$ of $f \circ \tau_0^{-1}$ and
  $g \circ \tau_\infty^{-1}$ for some
 representative $(f,g)$ of $[f,g]$.
 \end{proposition}
 \begin{proof}
  The last paragraph of
 the proof of Proposition
 \ref{pr:K_surjective_injective} shows that $L$ is a right inverse
 of $K \circ P^{-1}$, so long as it is well-defined.

 To show that $L$ is well-defined, assume that $\tilde{h}_1$ and $\tilde{h}_2$ are two possibly
 distinct quasiconformal extensions of   $f \circ \tau_0^{-1}$ and
  $g \circ \tau_\infty^{-1}$,  for a fixed
 representative $(f,g)$ of $[f,g]$.  In this case $\tilde{h}_1$ and
 $\tilde{h}_2$ must agree on $\partial A$, so $A_2=A_1$ and
 $(A,h_1,A_1)$ and $(A,h_2,A_1)$ are in the same equivalence class
 mod $\mathbb{Z}$ by Lemma \ref{le:when_boundary_equal}.  So $L$ is
 independent of the choice of extension.

 Now assume that $[f_1,g_1]=[f_2,g_2]$.  Then $f_2=af_1$ and
 $g_2=ag_1$ for some $a \in \mathbb{C}^*$.  If $\tilde{h}_i$
 are the corresponding extensions of $f_i \circ \tau_0^{-1}$ to
 $\cbar$ for $i=1,2$ and setting $A_i= \tilde{h}_i(A)$,
 then $aA_1=A_2$ and $a\tilde{h}_1=\tilde{h}_2$ on $\partial
 A$.  By  Lemma \ref{le:when_boundary_equal}
 \[  [A,\left. \tilde{h}_1 \right|_A,A_1]=[A,\left.  \tilde{h}_2
 \right|_A,A_2].  \]  Thus $L$ is well-defined.

 Since $K \circ P^{-1}$ is injective, and $L$ is a right inverse,
 it is also a left inverse.  So $K \circ P^{-1} = L^{-1}$.
 \end{proof}
 \begin{remark} \label{re:why_easier}
  The maps $L$ and $K \circ P^{-1}$ are analogous to the map
  identifying the set of non-overlapping maps $\Oqc(\riem)$
  with a fiber in $T(\riem)$ which we defined in
  \cite{RadSchip_fiber}.  However the results of that paper do not apply
  in this case, because there we used the fact that the automorphism group
  has no continuous subgroups in an essential way; this is false
  for an annulus. Furthermore,
  the base space reduces to a
  point, and the ``fiber'' becomes rather
  $\Oqc(\mathbb{C}^*)/\Aut(\cpunc)$.  On the other hand, in some
  ways the proof that $L$ is a biholomorphism is more transparent in
  the case at hand, because here $L$ has an {\it explicit} inverse.
 \end{remark}
\end{subsection}
\end{section}
\begin{section}{The complex structures and holomorphicity of
multiplication} \label{se:compatibility}
\begin{subsection}{The complex structure inherited from $T(A)$}
 Since $\mathrm{PModI}(A)\cong \mathbb{Z}$ acts properly discontinuously and fixed-point freely by
 biholomorphisms, it follows from Proposition
 \ref{pr:K_surjective_injective}
 that $\mathcal{A}^0$ inherits a complex structure from that of
 $T(A)$.  We have thus shown that
 \begin{theorem}  \label{th:complex_structure_from_TA}
  $\mathcal{A}^0$ possesses a complex structure inherited from
  $T(A)$.
 \end{theorem}
 Furthermore, by our previous work, in this complex
 structure, the multiplication is holomorphic.
 \begin{theorem} \label{th:mult_holo_TB_version}
  Multiplication in $\mathcal{A}^0$ in the complex structure
  inherited from $T(A)$ is holomorphic.
 \end{theorem}
 \begin{proof}  This follows from \cite[Theorem 6.7]{RadSchip05},
 with the choice $g_X=g_Y=0$, $n_X^-=n_Y^-=1$, $n_X^+=n_Y^+=1$.
 \end{proof}

\end{subsection}

\begin{subsection}{Compatibility of the two complex structures}
  To show that the two complex structures on
 $\Oqc(\cpunc)/\Aut(\cpunc)$ are compatible, we need to show that
 the maps $L$ and $L^{-1}=K \circ P^{-1}$ are holomorphic, where
 $\Oqc(\cpunc)/\Aut(\cpunc)$ is endowed with the complex structure
 inherited from $\Oqc(\cpunc)$.
 \begin{theorem} \label{th:L_holomorphic}
  $L$ is holomorphic.
 \end{theorem}
 \begin{proof}
  $L$ has a lift to $\Oqc(\cpunc)$
 given by
 \begin{eqnarray*}
  \tilde{L}: \Oqc(\cpunc) & \rightarrow & T(A)/\mathbb{Z} \\
  (f,g) & \mapsto & (A, \left. \tilde{h} \right|_{A} h(A))
 \end{eqnarray*}
 where $\tilde{h}$ is a quasiconformal extension of $f \circ
 \tau_0^{-1}$ and $g \circ \tau_\infty^{-1}$ (one exists by
 \cite[Corollary 4.1]{RadSchip05}).  It follows directly that
 $L = \tilde{L} \circ r_a$.  Since $r_a$ is biholomorphic, it suffices to
 show that $\tilde{L}$ is holomorphic in order to show that $L$ is.
 It furthermore suffices to show that in a neighbourhood of any
 point $P^{-1} \circ \tilde{L}$ is
 holomorphic for some local inverse of $P$.  The fundamental
 projection
 $\Phi:L^\infty_{-1,1}(A) \rightarrow T(A)$ has local holomorphic inverses
 in a neighbourhood of any point.  We will show that for any choice
 of local inverses $\Phi$ and $P$, $\Phi^{-1}
 \circ P^{-1} \circ \tilde{L}$ is G\^ateaux holomorphic and locally
 bounded.  This is sufficient to demonstrate holomorphicity \cite[p
 178]{Chae}.

 Fix a point $(f,g) \in \Oqc(\cpunc)$.  Let
 $B(f,g)=(u^0,q^0,u^\infty,q^\infty)$ where $B$ is the map
 defined by equation (\ref{eq:nice_chart}).  We will construct a
 holomorphic curve through $B(f,g)$.  Let
 $(v^0,c^0,v^\infty,c^\infty) \in A^1_\infty(\disk) \oplus
 \mathbb{C}^* \oplus A^1_\infty(\disk) \oplus \mathbb{C}^*$.
 Define the complex line $Y(t) = (u_t^0,q_t^0,u_t^\infty,q_t^\infty)$ by
$$ Y(t)= (u^0,q^0,u^\infty,q^\infty) +
 t(v^0,c^0,v^\infty,c^\infty) .
 $$
 The curve $B^{-1} \circ Y(t)$ has an explicit expression, which can
 be found by integrating the differential equation
 $\mathcal{A}(f_t)=u^0+tv^0$, $f_t'(0)=q_t^0$, and similarly for
 $g_t$.  The solution is
 \[  \psi_t=(f_t,g_t) \]
 where
 \[  f_t = \frac{q_t^0}{q^0} \int_0^z f'(\xi)
 \exp{\left(t \int_0^\xi v^0(w) dw \right)} d\xi \]
 and
 \[  S(g_t)= \frac{q_t^\infty}{q^\infty} \int_0^z S(g)'(\xi)
 \exp{\left(t \int_0^\xi v^\infty(w) dw \right)} d\xi.  \]
 We can recover $g_t$ from the second expression if desired by using
 the fact that $S(S(g_t))=g_t$.

 For fixed $z$, both expressions are holomorphic in $t$ by inspection.  By
 \cite[Theorem 3.3]{RadSchip_nonoverlapping} there is a
 neighbourhood $N$ of the origin in $\mathbb{C}$ such that
 $(f_t,g_t)$ are in $\Oqc(\cpunc)$ for all $t \in N$.  Thus the map
 of $f_0(\disk) \cup g_0(\disk^*)$ given by $f_t\circ f_0^{-1}$ restricted
 to $f_0(\disk)$ and $g_t \circ g_0^{-1}$
 restricted to $g_0(\disk^*)$ is a holomorphic motion.  By the extended
 lambda-lemma \cite{Slodkowski}, this extends to a holomorphic motion
 $\tilde{h}_t$ of $\cbar$.  Thus denoting $h_t= \left. \tilde{h_t}
 \right|_A$ and $A_t=h_t(A)$, we get a curve
 \[  t \mapsto (A,h_t,A_t)  \]
 in $T(A)$.

 Because $h_t$ is a holomorphic motion, its complex
 dilatation $\mu(h_t)$ is a holomorphic curve in
 $L^\infty_{-1,1}(A)$ (see for example \cite[Theorem 2]{BR86}).
 Clearly $\mu(h_t) = \Phi^{-1} \circ
 P^{-1} \circ \tilde{L}(f_t,g_t)$ for some local choices of $\Phi^{-1}$ and
 $P^{-1}$.  We have thus shown that $\Phi^{-1} \circ
 P^{-1} \circ \tilde{L}$ is G\^ateaux holomorphic for some local
 choices of $P^{-1}$ and $\Phi^{-1}$.

 Since $L^\infty_{-1,1}(A)_1$ is globally bounded by one, $\Phi^{-1} \circ
 P^{-1} \circ \tilde{L}$ is locally bounded.  This proves the claim.
 \end{proof}

 \begin{theorem} \label{th:Linverse_holomorphic}
  $L^{-1}$ is holomorphic.
 \end{theorem}
 \begin{proof}
  It is enough to show that $K$ is holomorphic, since $P:T(A)
  \rightarrow T(A)/\mathbb{Z}$ has local holomorphic sections
  and $L^{-1} = K \circ P^{-1}$.
  $K$ has a lift
  \begin{eqnarray*}
   \tilde{K}: T(A) & \rightarrow & \Oqc(\cpunc)\\
   \nonumber
   [A,h,A'] & \mapsto & (\tilde{h} \circ \tau_0,\tilde{h} \circ
   \tau_\infty)
  \end{eqnarray*}
  where we choose the unique representation $(\tilde{h} \circ
  \tilde{\tau}_0,\tilde{h} \circ \tilde{\tau}_\infty)$ of $K([A,h,A'])$ with the
  normalization $(\tilde{h} \circ \tilde{\tau}_{\infty})'(\infty)=1$.
  It suffices to show that
  $\tilde{K}$ is holomorphic.

  Fix $[A,h,A'] \in T(A)$, and let $(f,g)=\tilde{K}([A,h,A'])$.
  Let $B^0$ and $B^{\infty}$ be open sets containing $\overline{f(\disk)}$
  and $\overline{g(\disk^*)}$ respectively.  Let $V$ be a neighbourhood of
  $\Oqc(\cpunc)$ such that $\overline{f_1(\disk)} \subset B^0$
  and $\overline{g_1(\disk)} \subset B^\infty$ for all $(f_1,g_1) \in V$.
  Recall that the map $(f_1,g_1) \mapsto (f_1,S(g_1))$ is holomorphic
  by definition of the complex structure on $\Oqc(\cpunc)$.
  Let $\pi_1:\Oqc(\cpunc) \rightarrow \Oqc$ be the projection
  onto the first component $(f_1,g_1)
  \mapsto f_1$.  We will show that $\pi_1 \circ \tilde{K}$ is
  holomorphic.  The proof for the rigging at infinity is the same,
  except for the introduction of the map $S$.  (The normalization
  causes no difficulties).

  Let $t \mapsto [A,h_t,A_t]$ be a holomorphic curve in
  $T(A)$ for $t$ in a neighbourhood $N \subset \mathbb{C}$
  of zero.  Since the fundamental projection $\Phi:L^\infty_{-1,1}(A)
  \rightarrow T(A)$ is holomorphic and has holomorphic sections,
  we can assume without loss of generality that $t \mapsto \mu(h_t)$ is
  holomorphic.

  Note that $\tilde{h}$ in (Definition \ref{de:K_definition}) of $K$ is related to $h$ by a sewing operation and composition on the left by a holomorphic map. The normalization is also a left composition  by a holomorphic map.
  So by the holomorphicity of the sewing operation \cite[Lemma 6.3]{RadSchip05},
  $t \mapsto \mu(\tilde{h}_t)$ is a holomorphic map into
  $L^\infty_{-1,1}(\cbar)$. Because $\tilde{h}_t$ is normalized, it is
  the unique solution to the
  Beltrami equation with this normalization and with dilatation $\mu(\tilde{h}_t)$. Since solutions
  to the Beltrami equation depend holomorphically on $\mu$
  \cite[Proposition 4.7.6]{Hubbard}, we see that
  $\tilde{h}_t(z)$ is holomorphic in $t$ for fixed $z$.

   Denote $f_t= \tilde{h}_t \circ
  \tau_0$ and $g_t = \tilde{h}_t \circ \tau_\infty$.  Then $f_t(z)$
  is separately holomorphic in $t$ and $z$, so it is jointly
  holomorphic. Thus all of the $z$-derivatives of $f_t(z)$ are
  holomorphic in $t$ on $\mathbb{D}$.

  Thus, $\mathcal{A}(f_t)(z)$ is holomorphic in $t$ for all fixed
  $z$, and $f_t'(0)$ is holomorphic in $t$.  Define
  $E_z:A_1^\infty \rightarrow \mathbb{C}$ to be the point
  evaluation functional $E_z(\rho)=\rho(z)$.  These are continuous
  linear functionals for all $z \in \mathbb{D}$.  For any open
  subset $D\subset \mathbb{D}$, $G=\{ E_z : z \in D \}$ form a
  separating subset of the dual of $A_1^\infty$.  By the previous
  paragraph, $E_z(\mathcal{A}(f_t))= \mathcal{A}(f_t)(z)$ are
  holomorphic in $t$.  So by \cite{Grosse} (see also \cite[Theorem 3.8]{RadSchip_fiber}), if we show that $t \mapsto
  (\mathcal{A}(f_t),f_t'(0))$ is locally bounded, we will have shown that
  $\tilde{K}$ is G\^ateaux holomorphic.
  In fact, if we show that $\tilde{K}$ is locally bounded,
   we can further conclude that
  $\tilde{K}$ is holomorphic (see for example \cite[p 198]{Chae} or \cite[Theorem 3.7]{RadSchip_fiber}).

  We show that $\tilde{K}$ is bounded.  By \cite[Theorem 4.7.4]{Hubbard}
  applied to $\tilde{h} \circ \tau_0$ and $\tilde{h} \circ \tau_\infty$,
  there is a
  neighbourhood $W$ of $[A,h,A']$ in $T(A)$ such that
  $\tilde{K}([A,h_1,A_1]) \subset V$ for all $[A,h_1,A_1] \in V$.
  Now if $(f,g) \in V$, we have that $f(\mathbb{D}) \subset U^0$.
  Since the closure of $B^0$ does not contain $\infty$, we can
  assume that $B^0 \subset \{z : |z|<R\}$ for some $R>0$.  Thus
  $|f'(0)| \leq R$ by the Schwarz lemma.  By the second
  coefficient estimate for univalent functions, we have that
  \[  \left| (1-|z|^2) \Psi(f_1)(z) - 2 \bar{z} \right| \leq
  4  \]
  so
  \[  \| (1-|z|^2) \Psi(f_1)(z) \|_\infty \leq 6  \]
  for all $f_1 \in \pi_1(V)$.  A similar argument shows that $S(g)'(0)$
  and $\Psi(S(g))$ are bounded.  Thus $\tilde{K}$ is bounded on $W$.
  This completes the proof.
 \end{proof}

Finally, we observe that multiplication of annuli is holomorphic.
\begin{corollary} \label{co:multiplication_holomorphic}
 Multiplication of annuli is holomorphic, both in the complex
 structure of non-overlapping maps, and in the complex structure
 inherited from the Teichm\"uller space of annuli.
\end{corollary}
\begin{proof} By Theorems \ref{th:L_holomorphic} and
\ref{th:Linverse_holomorphic}, it is enough to show that
multiplication is holomorphic in the complex structure inherited by
the Teichm\"uller space of annuli.   Thus the claim follows from
\ref{th:mult_holo_TB_version}.
\end{proof}
It also follows immediately that
\begin{corollary}  The complex structure on $\mathcal{A}^0$
inherited from $T(A)$ does not depend on the choice of standard base
$(A,\tau_0,\tau_\infty)$.
\end{corollary}
\begin{remark} This also follows from \cite[Theorem
5.4]{RadSchip05}.
\end{remark}
\end{subsection}

\begin{subsection}{Summary of mappings}
 For the convenience of the reader, we
provide a diagram illustrating the maps identifying the various
spaces. The bottom row is from Theorem
\ref{th:complexstructure_via_Oqc}.

\begin{equation}
\label{teepee}
\xymatrix@+30pt{
& T(A) \ar[d]^P \ar@/^2pc/[ddr]^{K} & & \\
& T(A)/\mathbb{Z} \ar[dl]_{F}^{\simeq} \ar@/^0.7pc/[dr]^{K \circ P^{-1}}  & & \\
\mba \ar[r]_-{R}^-{\simeq} & \mathcal{A}^o \ar[r]_-{I}^-{\simeq}
& \Oqc(\cpunc)/\Aut(\cpunc) \ar@/^0.7pc/[ul]^L \ar[r]_-{r_a}^-{\simeq} & \Oqc(\cpunc)_a \subset \Oqc(\cpunc) }
\end{equation}

\end{subsection}

\begin{subsection}{A complex structure on bounded univalent functions}
 The bounded univalent functions are a subsemigroup of
 the semigroup of rigged annuli \cite{Neretin_holomorphic,Segal}.
 In this section, we show that
 the bounded univalent functions with quasiconformal extensions
 $\mathcal{E}^0$
 inherits a complex structure from $\mathcal{A}^0$ in which
 composition is holomorphic.
 \begin{theorem} \label{th:Esubmanifold}
   $\mathcal{E}^o$ is a complex submanifold of $\ann^o$.
 \end{theorem}
 \begin{proof}
  By the definition of the complex structure on $\ann^o$ and
  Corollary \ref{co:normalized_Oqc_submanifold}, it suffices to
  show that $\mathcal{E}^o$ is a complex submanifold of
  $\Oqc(\cpunc)$.  Fix a point $(f, \text{Id}) \in \mathcal{E}^o \subset
  \ann^o$.  We have a global embedding (Section
  \ref{se:complex_structure_definition})
  \[   B(f_1,g_1) = \left(
  \mathcal{A}(f_1),f_1'(0),\mathcal{A}(S(g_1)),g_1'(\infty) \right). \]
  For every
  element $(f_1,g_1) \in \mathcal{E}^o$,
  \[  B(f_1,g_1) = \left(\mathcal{A}(f_1), f_1'(0),0,1 \right). \]
  It follows
  that $\mathcal{E}^o$ is a complex submanifold of $\mathcal{A}^0$.
 \end{proof}
 \begin{corollary} \label{co:Emultiplication_holo}
  The set of normalized quasiconformally extendible bounded
  univalent maps $\mathcal{B}$ possesses a complex structure, in
  which composition is holomorphic.  This complex structure is
  compatible with that inherited from the Teichm\"uller space of the
  annulus $T(A)$.
 \end{corollary}
 \begin{proof}  By Corollary \ref{co:multiplication_holomorphic},
 multiplication is holomorphic.  By Proposition
 \ref{pr:E_multiplication}, multiplication in $\mathcal{E}^o$ is
 given by $(f_1, \text{Id}) \cdot (f_2, \text{Id})= (f_1 \circ f_2, \text{Id})$.
 \end{proof}

 This is an interesting application of the ideas of conformal field
 theory to geometric function theory.  As far as we know,
 composition in the semigroup of bounded univalent functions has not been
 related
 to the Teichm\"uller space of doubly-connected
 domains and its complex structure.

 Note that two fundamental semigroups
 in geometric function theory, quasisymmetries and bounded univalent
 functions, both under composition, are in some sense
 ``interpolated'' by the Neretin-Segal semigroup.
 On the other hand, normalized quasisymmetries are a model of the
 universal Teichm\"uller space, but composition of
 quasisymmetries is not
 even continuous in the universal Teichm\"uller space. There is no
 contradiction, since $\mathcal{G}$ is not a subset of $\ann^0$.
 It is of some interest to investigate whether $\mathcal{G}$
 is contained in an appropriately defined boundary of
 $\mathcal{A}^0$ (see \cite[444--445]{Segal}).
\end{subsection}
\end{section}

\end{document}